\documentclass{amsart}

\makeatletter
 \def\LaTeX{\leavevmode L\raise.42ex
   \hbox{\kern-.3em\size{\sf@size}{0pt}\selectfont A}\kern-.15em\TeX}
\makeatother

\newcommand{\BibTeX}{{\rm B\kern-.05em{\sc
i\kern-.025emb}\kern-.08em\TeX}}
\newtheorem{corollary}{Corollary}[section]
\newtheorem{theorem}{Theorem}[section]
\newtheorem{lemma}[theorem]{Lemma}

\newtheorem{remark}[theorem]{Remark}

\newtheorem{definition}[theorem]{Definition}

\numberwithin{equation}{section}

\begin{document}

\begin{center}

\title{Sampling formulas  for one-parameter groups of operators in Banach  spaces}

\maketitle

\author{Isaac Z. Pesenson }\footnote{ Department of  Mathematics, Temple University,
 Philadelphia,
PA 19122; pesenson@temple.edu  }

\end{center}

\bigskip



\begin{abstract}

We extend  some results about sampling of entire functions of exponential type to Banach spaces. By using  generator $D$ of one-parameter group $e^{tD}$ of isometries of a Banach space $E$ we introduce Bernstein subspaces $\mathbf{B}_{\sigma}(D),\>\>\sigma>0,$ of vectors $f$ in $E$ for which trajectories $e^{tD}f$ are  abstract-valued functions of exponential type which are  bounded on the real line. This property allows to reduce sampling problems for $e^{tD}f$ with $f\in \mathbf{B}_{\sigma}(D)$ to known sampling results for regular functions of exponential type $\sigma$.

\end{abstract}

\section{Introduction}

The goal of the paper is to extend  some theorems about sampling of entire functions of exponential type to Banach spaces. Our framework starts with  considering a generator $D$ of one-parameter strongly continuous  group  of operators $e^{tD}$ in a Banach space $E$.  The operator $D$ is used    to define analogs of  Bernstein subspaces $\mathbf{B}_{\sigma}(D)$. 
The main property of vectors $f$ in $\mathbf{B}_{\sigma}(D)$ is that corresponding  trajectories $e^{tD}f$ are  abstract-valued functions of exponential type which are  bounded on the real line. This fact  allows to apply known sampling theorems to every  function of the form $\left<e^{tD}f, \>g^{*}\right>\in \mathbb{B}_{\sigma}^{\infty}(\mathbb{R}),\>\>\>g^{*}\in E^{*},$ or of the form $t^{-1}\left<e^{tD}f-f, \>g^{*}\right>\in \mathbb{B}_{\sigma}^{2}(\mathbb{R})$, where $ \mathbb{B}_{\sigma}^{\infty}(\mathbb{R}),\>\> \mathbb{B}_{\sigma}^{2}(\mathbb{R})$ are classical Bernstein spaces on $\mathbb{R}$.

 In remark \ref{rem} we demonstrate that the assumption that $D$ generates a group of operators is somewhat essential if one wants to have non-trivial Bernstein spaces.  In section 2 we give a few  different descriptions of these spaces and one of them explores resent results in \cite{BSS} which extend classical Boas formulas  \cite{B1}, \cite{B} for entire functions of exponential type. Note, that in turn, Boas formulas are generalizations of  Riesz formulas \cite{R}, \cite{R1} for trigonometric polynomials.

Section \ref{reg} contains two sampling-type formulas  which explore regularly spaced samples and section \ref{irreg} is about two results in the spirit of irregular sampling. 
Section \ref{appl}  contains an application to inverse Cauchy problem for abstract Schr\"{o}dinger equation.

Note, that if $e^{tD} ,\>\>\>t\in \mathbb{R}, $ is a group of operators in a Banach space $E$ then any trajectory $e^{tD}f,\>\>\>f\in E,$ is completely determined by any (single) sample $e^{\tau D}f,$ because for any $t\in \mathbb{R}$
$$
e^{tD}f=e^{(t-\tau)D} \left(e^{\tau D}f\right).
$$

Our results  in sections \ref{reg} and \ref{irreg} have,  however, a different  nature. They represent a trajectory $e^{tD}f$ as a "linear combination" of a countable number of its samples. Such kind results can be useful when  the entire  group of operators $e^{tD}$ is unknown and only samples $e^{t_{k}D}f$ of a trajectory  $e^{tD}f$ are given.

It seems to be very interesting that not matter how complicated one-parameter group can be (think, for example, about a Schr\"{o}dinger operator $D=-\Delta+V(x)$ and the corresponding group $e^{itD}$ in $L_{2}(\mathbb{R}^{d})$) the formulas (\ref{s1}), (\ref{l0}), (\ref{s3}), (\ref{l3000}), (\ref{s4}) are universal in the sense that they contain the same coefficients and the same sets of sampling points.

It was my discussions  with Paul Butzer and Gerhard Schmeisser during Sampta 2013 in Jacobs University in Bremen of their beautiful work with Rudolf Stens \cite{BSS}, that stimulated my interest in the topic of the present paper. I am very grateful to them for this.
\section{Bernstein vectors  in Banach spaces}

We assume that $D$ is a generator of one-parameter group of
isometries $e^{ tD}$ in a Banach space $E$ with the norm $\|\cdot
\|$ (for precise definitions see \cite{BB}, \cite{K}). The notations $\mathcal{D}^{k}$ will be used for  the domain of $D^{k}$, and notation $\mathcal{D}^{\infty}$ for $\bigcap_{k\in \mathbb{N}}\mathcal{D}^{k}$.

\begin{definition}
The  subspace of exponential vectors $\mathbf{E}_{\sigma}(D), \>\>\sigma\geq 0,$ is defined as  a set of all vectors $f$ in  $\mathcal{D}^{\infty}$  for which there exists a constant $C(f,\sigma)>0$ such that 
\begin{equation}\label{Exp}
\|D^{k}f\|\leq C(f,\sigma)\sigma^{k}, \>\>k\in \mathbb{N}.
\end{equation}
 
\end{definition}
Note, that every $\mathbf{E}_{\sigma}(D)$ is clearly a linear subspace of $E$. What is really  important is the fact that union of all $\mathbf{E}_{\sigma}(D)$ is dense in $E$ (Corollary \ref{Density}). 

\begin{remark}\label{rem}
It is worth to stress  that if $D$ generates a strongly continuous bounded semigroup then the set  $\bigcup_{\sigma\geq 0}\mathbf{E}_{\sigma}(D)$ may not be  dense in $E$. 

Indeed,  consider a strongly continuous bounded semigroup $T(t)$ in $L_{2}(0,\infty)$ defined for every $f\in L_{2}(0,\infty)$ as $T(t)f(x)=f(x-t),$ if $\>x\geq t$ and $T(t)f(x)=0,$ if $\>\>0\leq x<t$.   Inequality (\ref{Exp}) implies that if $ f\in \mathbf{E}_{\sigma}(D)$ then for any $g\in L_{2}(0,\infty)$ the function $\left<T(t)f,\>g\right>$ is analytic in $t$. Thus if $g$ has compact support then $\left<T(t)f,\>g\right>$ is zero for  all  $t$ which implies that $f$ is zero. In other words in this case every space $\mathbf{E}_{\sigma}(D)$ is trivial. 

\end{remark}

\begin{definition}
The Bernstein subspace
 $\mathbf{B}_{\sigma}(D), \>\>\sigma\geq 0,$ is defined as  a set of all vectors $f$ in $E$ 
 which belong to $\mathcal{D}^{\infty}$  and for which 
\begin{equation}\label{Bernstein}
\|D^{k}f\|\leq \sigma^{k}\|f\|, \>\>k\in \mathbb{N}.
\end{equation}
\end{definition}

\begin{lemma}[\cite{Pes00}]\label{basic}
  Let $D$ be a generator of an one parameter group of operators
$e^{tD }$ in a Banach space $E$ and $\|e^{tD}f\|=\|f\|$.  If for
some $f\in E$ there exists an $\sigma
>0$ such
that the quantity

$$
\sup_{k\in N} \|D^{k}f\|\sigma ^{-k}=R(f,\sigma)
$$
 is finite, then $R(f,\sigma)\leq
\|f\|.$
\end{lemma}

\begin{proof}
  By assumption $\|D^{r}f\|\leq R(f,\sigma)\sigma ^{r}, r\in \mathbb{N}$.  Now for
any complex number $z$ we have

$$ \left\|e^{zD}f\right\|=\left\|\sum ^{\infty}_{r=0}(z^{r}D^{r}g)/r!\right\|\leq R(f,\sigma) \sum
^{\infty}_{r=0}|z|^{r}\sigma^{r}/r!=R(f,\sigma)e^{|z|\sigma}.$$

It implies that for any functional $h^{*}\in E^{*}$ the scalar function
$\left<e^{zD}f,h^{*}\right>$ is an entire function
 of exponential type $\sigma $ which is bounded on the real axis  $\mathbb{R}$
by the constant $\|h^{*}\| \|f\|$.
 An application of the Bernstein inequality gives

$$\left\|<e^{tD}D^{k}f,h^{*}>\right\|_{C(\mathbb{R})}=\left\|\left(\frac{d}{dt}\right)^{k}\left<e^{tD}f,h^{*}\right>\right\|_{
C(\mathbb{R})} \leq\sigma^{k}\|h^{*}\| \|f\|.$$

The last one gives for $t=0$

$$ \left|\left<D^{k}f,h^{*}\right>\right|\leq \sigma ^{k} \|h^{*}\| \|f\|.$$

Choosing $h^{*}$ such that $\|h^{*}\|=1$ and $\left<D^{k}f,h^{*}\right>=\|D^{k}f\|$ we obtain
the inequality $\|D^{k}f\|\leq
 \sigma ^{k} \|f\|, k\in N$, which gives

$$R(f,\sigma)=\sup _{k\in \mathbb{N} }(\sigma ^{-k}\|D^{k}f\|)\leq \|f\|.$$

Lemma is proved.
\end{proof}

\begin{theorem}[\cite{Pes14}]\label{t1}
  Let $D$ be a generator of  one-parameter group of operators
$e^{tD }$ in a Banach space $E$ and $\|e^{tD}f\|=\|f\|$.  Then for every $\sigma\geq 0$ 
$$
\mathbf{B}_{\sigma}(D)= \mathbf{E}_{\sigma}(D) , \>\>\>\sigma\geq 0, 
$$

\end{theorem}

\begin{proof}
The inclusion  $
\mathbf{B}_{\sigma}(D)= \mathbf{E}_{\sigma}(D) , \>\>\>\sigma\geq 0, 
$ is obvious. The opposite inclusion follows from the previous Lemma. 
\end{proof}

Motivated by results in \cite{BSS} we introduce the following  bounded operators

\begin{equation}\label{b1}
\mathcal{B}_{D}^{(2m-1)}(\sigma)f=\left(\frac{\sigma}{\pi}\right)^{2m-1}\sum_{k\in \mathbb{Z}}(-1)^{k+1}A_{m,k}e^{\frac{\pi}{\sigma}(k-1/2)D}f,\>\> f\in E, \>\>\sigma>0,\>\>\>m\in \mathbb{N},
\end{equation}
\begin{equation}\label{b2}
\mathcal{B}_{D}^{(2m)}(\sigma)f=\left(\frac{\sigma}{\pi}\right)^{2m}\sum_{k\in \mathbb{Z}}(-1)^{k+1}B_{m,k}e^{\frac{\pi k}{\sigma}D}f, \>\>f\in E,\> \sigma>0,\>m\in \mathbb{N},
\end{equation}

where 
\begin{equation}\label{A}
A_{m,k}=(-1)^{k+1} \rm sinc ^{(2m-1)}\left(\frac{1}{2}-k\right)=
$$
$$
\frac{(2m-1)!}{\pi(k-\frac{1}{2})^{2m}}\sum_{j=0}^{m-1}\frac{(-1)^{j}}{(2j)!}\left(\pi(k-\frac{1}{2})\right)^{2j},\>\>\>m\in \mathbb{N},
\end{equation}
for $k\in \mathbb{Z}$ and 
\begin{equation}\label{B}
B_{m,k}=(-1)^{k+1} \rm sinc ^{(2m)}(-k)=\frac{(2m)!}{\pi k^{2m+1}}\sum_{j=0}^{m-1}\frac{(-1)^{j}(\pi k)^{2j+1}}{(2j+1)!},\>\>\>m\in \mathbb{N},\>\>\>k\in \mathbb{Z}\setminus {0},
\end{equation}
and 
\begin{equation}\label{B0}
B_{m,0}=(-1)^{m+1} \frac{\pi^{2m}}{2m+1},\>\>\>m\in \mathbb{N}.
\end{equation}

  Both series converge in $E$ due to the following formulas (see \cite{BSS})

\begin{equation}
\left(\frac{\sigma}{\pi}\right)^{2m-1}\sum_{k\in \mathbb{Z}}\left|A_{m,k}\right|=\sigma^{2m-1},\>\>\>\>\>
 \left(\frac{\sigma}{\pi}\right)^{2m}\sum_{k\in \mathbb{Z}}\left|B_{m,k}\right|=\sigma^{2m}\label{id-2}.
 \end{equation}
 Since $\|e^{tD}f\|=\|f\|$ it implies that 
 
 \begin{equation}\label{norms}
 \|\mathcal{B}_{D}^{(2m-1)}(\sigma)f\|\leq \sigma^{2m-1}\|f\|,\>\>\>\>\>\|\mathcal{B}_{D}^{(2m)}(\sigma)f\|\leq \sigma^{2m}\|f\|,\>\>\>f\in E. 
 \end{equation}
 
 For the following theorem see  \cite{Pes00}, \cite{Pes08}, \cite{Pes11}, \cite{Pes14}.

 \begin{theorem}
 If $D$ generates a one-parameter strongly continuous bounded group of operators $e^{tD}$ in a Banach space $E$ then  the following conditions are equivalent:

\begin{enumerate}

\item  $f$ belongs to $\mathbf{B}_{\sigma}(D)$.

\item   The abstract-valued function $e^{tD}f$  is entire abstract-valued function of exponential type $\sigma$ which is bounded on the real line.

\item For every functional $g^{*}\in E^{*}$ the function $\left<e^{tD}f,\>g^{*}\right>$  is entire function of exponential type $\sigma$ which is bounded on the real line.
\item The following Boas-type interpolation formulas hold true for $r\in \mathbb{N}$

\begin{equation}\label{B1}
D^{r}f=\mathcal{B}_{D}^{(r)}(\sigma)f,\>\>\>\>\>f\in \mathbf{B}_{\sigma}(D).
\end{equation}

\end{enumerate}

\end{theorem}

\begin{corollary}
Every $\mathbf{B}_{\sigma}(D)$ is a closed linear subspace of $E$.
\end{corollary}

 Let's introduce the operator
 $\>\>
 \Delta^{m}_{s}f=(I-e^{sD})^{m}f, \>\>m\in \mathbb{N}, 
 $
 and  the modulus of
continuity \cite{BB}
$$
\Omega_{m}(f,s)=\sup_{|\tau|\leq
s}\left\|\Delta^{m}_{\tau}f\right\|\label{dif}.
$$

The following theorem is proved in \cite{Pes09}, \cite{Pes11},  \cite{Pes14}.
\begin{theorem} There exists a constant $C>0$ such that for all
$\sigma>0$ and all $f\in \mathcal{D}^{k}$
\begin{equation}
inf_{ g\in \mathbf{B}_{\sigma}(D) }\|f-g\|\leq
C\sigma^{-k}\Omega_{m-k}\left(D^{k}f, \sigma^{-1}\right),
0\leq k\leq m.\label{J}
\end{equation}
\end{theorem}
\begin{corollary}\label{Density}
The set $\bigcup_{\sigma\geq 0}\mathbf{B}_{\sigma}(D)$ is dense in $E$. 
\end{corollary}

\begin{definition}
For a given $f\in E$ the notation $\sigma_{f}$ will be used for the smallest finite real number (if any) for which 
$$\|D^{k}f\|\leq \sigma_{f}^{k}\|f\|,\>\>\>k\in\mathbb{N}.
$$
 If there is no such finite number we assume that $\sigma_{f}=\infty$. 
\end{definition}

Now we  are going to prove another characterization of Bernstein spaces. In the case of Hilbert spaces corresponding result was proved in \cite{Pes08a}, \cite{PZ}.

\begin{theorem}\label{new}
Let $f\in E$ belongs to a space  $\mathbf{B}_{\sigma}(D),$ for some
$0<\sigma<\infty.$ Then the following limit exists
 \begin{equation}
 d_f=\lim_{k\rightarrow \infty} \|D^k
f\|^{1/k} \label{limit}
\end{equation}
and $d_f=\sigma_f.$ 

Conversely, if
$f\in \mathcal{D}^{\infty}$ and $d_f=\lim_{k\rightarrow \infty}
\|D^k f\|^{1/k},$ exists and is finite, then $f\in\mathbf{B}_{d_f}(D)$
and $d_f=\sigma_f .$

\end{theorem}

Let us introduce the Favard constants (see \cite{Akh}, Ch. V)
which are defined as
$$
K_{j}=\frac{4}{\pi}\sum_{r=0}^{\infty}\frac{(-1)^{r(j+1)}}{(2r+1)^{j+1}},\>\>\>
j,\>\>r\in \mathbb{N}.
$$
It is known \cite{Akh}, Ch. V, that the sequence of all Favard
constants with even indices is strictly increasing and belongs to
the interval $[1,4/ \pi)$ and the sequence of all Favard constants
with odd indices is strictly decreasing and belongs to the
interval $(\pi/4, \pi/2],$ i.e.,
\begin{equation}
K_{2j}\in [1,4/ \pi), \; K_{2j+1}\in (\pi/4, \pi/2].\label{Fprop}
\end{equation}
We will need the following generalization of the classical
Kolmogorov inequality. It is worth noting that the inequality was
first proved by Kolmogorov for $L^\infty (\mathbb{R})$ and later extended
to $L^p(\mathbb{R})$ for $1\leq p < \infty$ by Stein \cite{Stein} and that
is why it is known as the Stein-Kolmogorov inequality.
\begin{lemma} Let $f\in \mathcal{D}^{\infty} .$ Then,
the following inequality holds
\begin{equation}
\left\|D^{k}f\right\|^n \leq C_{k,n}\|D^{n}f\|^{k}\|f\|^{n-k},\>\>\>
0\leq k \leq n,\label{KS}
\end{equation}
where $C_{k,n}= (K_{n-k})^n/(K_{n})^{n-k}.$ \label{LKS}
\end{lemma}
\begin{proof}
Indeed, for any $h^{*}\in E^{*}$ the Kolmogorov inequality \cite {Stein}
applied to the entire function $\left<e^{tD}f,h^{*}\right>$ gives
$$
\left\|\left(\frac{d}{dt}\right)^{k}\left<e^{tD}f,h^{*}\right>\right\|^n_{
C(\mathbb{R}^{1})}\leq
C_{k,n}\left\|\left(\frac{d}{dt}\right)^{n}\left<e^{tD}f,h^{*}\right>\right\|_{
C(\mathbb{R}^{1})}^{k}\times
$$
$$
\left\|\left<e^{tD}f,h^{*}\right>\right\|_{C(\mathbb{R}^{1})}^{n-k},\quad
0<k< n,
$$
or
$$
\left\|\left<e^{tD}D^{k}f,h^{*}\right>\right\|_{ C(\mathbb{R}^{1})}^n\leq
C_{k,n}\left\|\left<e^{tD}D^{n}f,h^{*}\right>\right\|_{
C(\mathbb{R}^{1})}^{k}\left\|\left<e^{tD}f,h^{*}\right>\right\|_{C(\mathbb{R}^{1})}^{n-k}.
$$
Applying the Schwartz inequality to the right-hand side, we obtain
\begin{align*}
\left\|\left<e^{tD}D^{k}f,h^{*}\right>\right\|_{ C(\mathbb{R}^{1})}^n & \leq
C_{k,n}\|h^{*}\|^{k}\|D^{n}f\|^{k}\|h^{*}\|^{n-k}\|f\|^{n-k}\\
&\leq C_{k,n}\|h^{*}\|^n \|D^{n}f\|^{k}\|f\|^{n-k},
\end{align*}
which, when $t=0,$ yields
$$
\left|\left<D^{k}f,h^{*}\right>\right|^n\leq C_{k,n}\|h^{*}\|^n
\|D^{n}f\|^{k}\|f\|^{n-k}.
$$
By choosing $h$ such that $\left|<D^{k}f,h^{*}>\right|=\|D^{k}f\|$ and
$\|h^{*}\|=1$  we obtain (\ref{KS}).
\end{proof}

{\bf Proof of Theorem \ref{new}.}

From Lemma \ref{LKS} we have $$ \left\|D^{k}f\right\|^n \leq
C_{k,n}\|D^{n}f\|^{k}\|f\|^{n-k}, \quad 0\leq k \leq n.
$$
Without loss of generality, let us assume that $\|f\|=1.$ Thus,
$$ \left\|D^{k}f\right\|^{1/k} \leq
(\pi/2)^{1/kn}\|D^{n}f\|^{1/n}, \quad 0\leq k \leq n.
$$
Let $k$ be arbitrary but fixed. It follows that
$$\left\|D^{k}f\right\|^{1/k} \leq
(\pi/2)^{1/kn}\|D^{n}f\|^{1/n}, \mbox{ for all } n\geq k,$$ which
implies that
$$\left\|D^{k}f\right\|^{1/k}\leq \underline{\lim}_{n\rightarrow
\infty}\|D^{n}f\|^{1/n}.$$ But \rm since this inequality is true for
all $k>0,$ we obtain that
$$\overline{\lim}_{k\rightarrow\infty}\|D^{k}f\|^{1/k}\leq
\underline{\lim}_{n\rightarrow \infty}\|D^{n}f\|^{1/n},$$ which
proves that $d_f=\lim_{k\rightarrow}\|D^{k}f\|^{1/k}$ exists.

Since $f\in \mathbf{B}_{\sigma}(D) $ the constant $\sigma_{f}$ is finite and we have 
$$
\|D^{k}f\|^{1/k}\leq \sigma_f  \|f\|^{1/k},
$$
 and by taking the
limit as $k\rightarrow\infty$ we obtain $d_f\leq \sigma_f.$ To show
that $d_f= \sigma_f,$ let us assume that $d_f< \sigma_f.$
Therefore, there exist $M>0$ and $\sigma $ such that $0<d_f<\sigma
< \sigma_f$ and $$\|D^k f\|\leq M \sigma^k, \quad \mbox{for all }
k>0 .$$ Thus, by Lemma  \ref{basic} we have $f\in \mathbf{B}_\sigma(D) ,$
which is a contradiction to the definition of $\sigma_f.$

Conversely,  suppose that $d_f=\lim_{k\rightarrow \infty} \|D^k
f\|^{1/k}$ exists and is finite. Therefore, there exist $M>0$ and
$\sigma
>0$ such that $d_f<\sigma $ and
$$\|D^k f\|\leq M \sigma^k, \quad \mbox{for all } k>0 ,$$
which, in view of Lemma \ref{basic}, implies that $f\in
\mathbf{B}_{\sigma}(D) .$ Now by repeating the argument in the first part of
the proof we obtain $d_f=\sigma_f ,$ where $\sigma_f=\inf\left\{
\sigma : f\in \mathbf{B}_{\sigma}(D)\right\}.$

Theorem \ref{new} is proved.

Now consider the following abstract Cauchy problem for the operator $D$. 
\begin{equation}
\frac{d u(t)}{d t}=Du(t), \>\>u(0)=f,\label{Cauchy2}
\end{equation}
where $u: \mathbb{R} \rightarrow E$ is an abstract function with values in
$E.$
Since solutions of this problem given by the formula $u(t)=e^{tD}f$ we obtain the following result.

\begin{theorem}
A vector $f\in E,$ belongs to $\mathbf{B}_{\sigma}(D)$  if and only if the
solution $u(t)$ of the corresponding Cauchy problem
(\ref{Cauchy2}) has the following properties:

1) as a function of $t,$ it has an analytic extension $u(z), z\in
\mathbb{C}$ to the complex plane $\mathbb{C}$ as an entire
function;

2) it has exponential type $\sigma$ in the variable $z$, that is
$$
\|u(z)\|_{E}\leq e^{\sigma|z|}\|f\|_{E}.
$$ \label{MainThm}
and it is bounded on the real line. \label{PW2}
\end{theorem}

\section{Sampling-type formulas for one-parameter groups}\label{reg}

We assume that $D$ generates one-parameter strongly cononuous bounded group of operators $e^{tD}, \>\>t\in \mathbb{R},$ in a Banach space $E$. In this section we prove explicit formulas for a  trajectory $e^{tD}f$ with $f\in \mathbf{B}_{\sigma}(D)$ in terms of a countable number of equally spaced samples. 

\begin{theorem}
If $f\in \mathbf{B}_{\sigma}(D)$ then the following sampling formulas hold for $t\in \mathbb{R}$
\begin{equation}\label{s1}
e^{tD}f=f+tDf \rm sinc\left(\frac{\sigma t}{\pi}\right)+t\sum_{k\neq 0}\frac{e^{\frac{k\pi}{\sigma}D}f-f}{\frac{k\pi}{\sigma}} \rm sinc\left(\frac{\sigma t}{\pi}-k\right),
\end{equation}
 
\begin{equation}\label{l0}
f=e^{tD}f-t \left(e^{t D}Df\right)\rm sinc\left(\frac{\sigma t}{\pi}\right)-t\sum_{k\neq 0}\frac{e^{\left(   \frac{k\pi}{\sigma}+t  \right)D}f-e^{tD}f}{\frac{k\pi}{\sigma}}    \rm sinc\left(\frac{\sigma t}{\pi}+k\right).
\end{equation}
\end{theorem}
\begin{remark}

It is worth to note that if $\>\>\>t\neq  0,\>\>$ then right-hand side of (\ref{l0}) does not contain vector $f$ and we obtain a "linear combination"  of $f$ in terms of vectors  $e^{\left(   \frac{k\pi}{\sigma}+t  \right)D}f,\>\>\>k\in \mathbb{Z},$ and $e^{tD}Df$. 

\end{remark}
       \begin{proof}
       
         If $f\in \mathbf{B}_{\sigma}(D)$ then for any $g^{*}\in E^{*}$ the function $F(t)=\left<e^{tD}f,\>g^{*}\right>$ belongs to $B_{\sigma}^{\infty}(\mathbb{R})$.

We consider $F_{1}\in B_{\sigma}^{2}( \mathbb{R}),$ which is defined as follows.
If $t\neq 0$ then 
$$
F_{1}(t)=\frac{F(t)-F(0)}{t}=\left<\frac{e^{tD}f-f}{t},\>g^{*}\right>,
$$
and if $t=0$ then
$
F_{1}(t)=\frac{d}{dt}F(t)|_{t=0}=\left<Df,\>g^{*}\right>.
$
We have
$$
F_{1}(t)=\sum_{k}F_{1}\left(\frac{k\pi}{\sigma}\right)\> \rm sinc\left(\frac{\sigma t}{\pi}-k\right),
$$
which means that   for any $g^{*}\in E^{*}$ 
$$
\left<    \frac{e^{tD}f-f}{t},\>g^{*}    \right>=\sum_{k}\left<\frac{e^{\frac{k\pi}{\sigma}D}f-f}{\frac{k\pi}{\sigma}},\>g^{*}\right>\> \rm sinc\left(\frac{\sigma t}{\pi}-k\right).
$$
Since 
$$
\rm sinc ^{(n)} x=\sum_{j=0}^{n}C_{n}^{j}(\sin\>\pi x)^{(j)}\left(\frac{1}{\pi x}\right)^{(n-j)}=
$$
$$
\frac{(-1)^{n}n!}{\pi x^{n+1}}\sum_{j=0}^{n}\sin\left(\pi x+\frac{jx}{2}\right)\frac{(-1)^{j}(\pi x)^{j}}{j!}
$$
one has the estimate 
$$
|\rm sinc^{(n)}\>x|\leq \frac{C}{|x|},\>\>\>n=0,1,....
$$
which implies convergence in $E$ of the series 
$$
\sum_{k}\frac{e^{\frac{k\pi}{\sigma}D}f-f}{\frac{k\pi}{\sigma}} \rm sinc\left(\frac{\sigma t}{\pi}-k\right).
$$
It leads to the equality for any $g^{*}\in E^{*}$ 
$$
\left<    \frac{e^{tD}f-f}{t},\>g^{*}    \right>=\left<\sum_{k}\frac{e^{\frac{k\pi}{\sigma}D}f-f}{\frac{k\pi}{\sigma}} \rm sinc\left(\frac{\sigma t}{\pi}-k\right),\>g^{*}\right>,\>\>\>t\neq 0,
$$
and if $t= 0$ it gives the identity
$
\left<Df,\>g^{*}\right>=\left<Df,\>g^{*}\right>\sum_{k} \rm sinc \>k.
$
Thus,
\begin{equation}
   \frac{e^{tD}f-f}{t}=\sum_{k}\frac{e^{\frac{k\pi}{\sigma}D}f-f}{\frac{k\pi}{\sigma}} \rm sinc\left(\frac{\sigma t}{\pi}-k\right),\>\>\>t\neq 0,
\end{equation}
or for every $t\in \mathbb{R}$ 
$$
e^{tD}f=f+tDf \rm sinc\left(\frac{\sigma t}{\pi}\right)+t\sum_{k\neq 0}\frac{e^{\frac{k\pi}{\sigma}D}f-f}{\frac{k\pi}{\sigma}} \rm sinc\left(\frac{\sigma t}{\pi}-k\right).
$$

Thus, (\ref{s1}) is proved. 

If in (\ref{s1}) we replace $f$ by $\left(e^{\tau D}f\right)$ for a $\tau\in \mathbb{R}$ we will have

\begin{equation}\label{s100}
e^{tD}\left(e^{\tau D}f\right)=
$$
$$
\left(e^{\tau D}f\right)+t\sum_{k\neq 0}\frac{e^{\frac{k\pi}{\sigma}D}\left(e^{\tau D}f\right)-\left(e^{\tau D}f\right)}{\frac{k\pi}{\sigma}} \rm sinc\left(\frac{\sigma t}{\pi}-k\right)+tD\left(e^{\tau D}f\right)\rm sinc\left(\frac{\sigma t}{\pi}\right).
\end{equation}
For $t=-\tau $ we obtain the next formula which holds for any $\tau\in \mathbb{R},\>\>f\in \mathbf{B}_{\sigma}(D),$ 
\begin{equation}\label{l1}
f=e^{\tau D}f-\tau \sum_{k\neq 0}\frac{e^{\left(   \frac{k\pi}{\sigma}+\tau  \right)D}f-e^{\tau D}f}{\frac{k\pi}{\sigma}}    \rm sinc\left(\frac{\sigma \tau}{\pi}+k\right)-\tau D\left(e^{\tau D}f\right)\rm sinc\left(\frac{\sigma \tau}{\pi}\right),
\end{equation}
which is the formula (\ref{l0}). 
Theorem is proved. 
\end{proof}

The next Theorem is a generalization of what is known as Valiron-Tschakaloff
sampling/interpolation formula \cite{BFHSS}.

\begin{theorem}
For $f\in \mathbf{B}_{\sigma}(D),\>\>\>\sigma>0,$ we have for all $z \in \mathbb{C}$
\begin{equation}\label{VT}
e^{zD}f=z \>\rm sinc\left(\frac{\sigma  z}{\pi}\right)Df+\rm sinc\left(\frac{\sigma  z}{\pi}\right)f+\sum_{k\neq 0}\frac{\sigma z}{k\pi}\rm sinc\left(\frac{\sigma  z}{\pi}-k\right)e^{\frac{k\pi}{\sigma}D}f
\end{equation}
\end{theorem}

\begin{proof}
If $F\in \mathbf{B}_{\sigma}(D),\>\>\>\sigma>0,$ then for all $z \in \mathbb{C}$ the following 
Valiron-Tschakaloff sampling/interpolation formula holds \cite{BFHSS}

\begin{equation}
F(z)=z \>\rm sinc\left(\frac{\sigma  z}{\pi}\right)F^{'}(0)+\rm sinc\left(\frac{\sigma  z}{\pi}\right)F(0)+\sum_{k\neq 0}\frac{\sigma z}{k\pi}\rm sinc\left(\frac{\sigma  z}{\pi}-k\right)F\left(\frac{k\pi}{\sigma}\right)
\end{equation}
For $f\in  \mathbf{B}_{\sigma}(D),\>\>\>\sigma>0$ and $g^{*}\in E^{*}$ we have $F(z)=\left<e^{zD}f,\>g^{*}\right>\in  \mathbf{B}_{\sigma}(D)$ and thus

\begin{equation}
\left<e^{zD}f,\>g^{*}\right>=z \>\rm sinc\left(\frac{\sigma  z}{\pi}\right)\left<Df,\>g^{*}\right>+
$$
$$
\rm sinc\left(\frac{\sigma  z}{\pi}\right)\left<f,\>g^{*}\right>+\sum_{k\neq 0}\frac{\sigma z}{k\pi}\rm sinc\left(\frac{\sigma  z}{\pi}-k\right)\left<e^{\frac{k\pi}{\sigma}D}f,\>g^{*}\right>.
\end{equation}
Since the series 
$$
\sum_{k\neq 0}\frac{\sigma z}{k\pi}\rm sinc\left(\frac{\sigma  z}{\pi}-k\right)e^{\frac{k\pi}{\sigma}D}f
$$
converges in $E$ for every fixed $z$ we obtain the formula (\ref{VT}).

\end{proof}

\begin{theorem}
If $f\in \mathbf{B}_{\sigma}(D)$ then the following sampling formula holds for $t\in \mathbb{R}$ 
and  $n\in \mathbb{N}$
\begin{equation}\label{s2}
e^{tD}D^{n}f=\sum_{k}\frac{e^{\frac{k\pi}{\sigma}D}f-f}{\frac{k\pi}{\sigma}}\left\{  n \>\rm sinc^{(n-1)}\left(\frac{\sigma t}{\pi}-k\right)  +\frac{\sigma t}{\pi} \rm sinc^{(n)}\left(\frac{\sigma t}{\pi}-k\right)
   \right\}.
\end{equation}
In particular, for $n\in \mathbb{N}$
one has 
\begin{equation}\label{Q}
D^{n}f=\mathcal{Q}_{D}^{n}(\sigma)f,
\end{equation}
where the bounded operator $\mathcal{Q}_{D}^{n}(\sigma)$ is given by the formula
\begin{equation}
\mathcal{Q}_{D}^{n}(\sigma)f=n\sum_{k}\frac{e^{\frac{k\pi}{\sigma}D}f-f}{\frac{k\pi}{\sigma}}\left[   \rm sinc^{(n-1)}\left(-k\right)  +\rm sinc^{(n)}\left(-k\right)
   \right].
\end{equation}
\end{theorem}

\begin{proof}
Because $F_{1}\in B_{\sigma}^{2}(\mathbb{R})$ we have
$$
\left(\frac{d}{dt}\right)^{n}F_{1}(t)=\sum_{k}F_{1}\left(\frac{k\pi}{\sigma} \right) \rm sinc^{(n)}\left(\frac{\sigma t}{\pi}-k\right)
$$
and \rm since 
$$
\left(\frac{d}{dt}\right)^{n}F(t)=n\left(\frac{d}{dt}\right)^{n-1}F_{1}(t)+t\left(\frac{d}{dt}\right)^{n}F_{1}(t)
$$
we obtain 
$$
\left(\frac{d}{dt}\right)^{n}F(t)=n\sum_{k}F_{1}\left(\frac{k\pi}{\sigma}\right)\> \rm sinc^{(n-1)}\left(\frac{\sigma t}{\pi}-k\right)+\frac{\sigma t}{\pi}\sum_{k}F_{1}\left(\frac{k\pi}{\sigma}\right)\> \rm sinc^{(n)}\left(\frac{\sigma t}{\pi}-k\right)
$$
Because 
$
\left(\frac{d}{dt}\right)^{n}F(t)=\left<D^{n}e^{tD}f, g^{*}\right>,
$
and
$$
F_{1}\left(\frac{k\pi}{\sigma}\right)=\left<\frac{e^{\frac{k\pi}{\sigma}D}f-f}{\frac{k\pi}{\sigma}}, g^{*}\right>
$$
we obtain  that  for $t\in \mathbb{R},\>\>n\in \mathbb{N},$
$$
D^{n}e^{tD}f=\sum_{k}\frac{e^{\frac{k\pi}{\sigma}D}f-f}{\frac{k\pi}{\sigma}}\left[  n\> \rm sinc^{(n-1)}\left(\frac{\sigma t}{\pi}-k\right)  +\frac{\sigma t}{\pi} \rm sinc^{(n)}\left(\frac{\sigma t}{\pi}-k\right)
   \right].
$$ 
Theorem is proved. 
\end{proof}

\section{Irregular sampling theorems}\label{irreg}

In \cite{Hig} the following fact was proved.

\begin{theorem}\label{Hig}
Let $\{t_{n}\}_{n\in \mathbb{Z}}$ be a sequence of real numbers such that 
$$
\sup_{n\in\mathbb{Z}}|t_{n}-n|<1/4.
$$
Define the entire function 
$$
G(z)=(z-t_{0})\prod_{n=1}^{\infty}\left(1-\frac{z}{t_{n}}\right)\left(1-\frac{z}{t_{-n}}\right).
$$
 Then for all $f\in \mathbf{B}_{\pi}^{2}(\mathbb{R})$ we have
$$
f(t)=\sum_{n\in \mathbb{Z}}f(t_{n})\frac{G(t)}{G^{'}(t_{n})(t-t_{n})}
$$
uniformly on all compact subsets of $\mathbb{R}$.
\end{theorem}

An analog of this result for Banach spaces is the following. 

 \begin{theorem}
 If $D$ generates a one-parameter strongly continuous  group  $e^{tD}$ of isometries  in a Banach space $E$.  Suppose  that assumptions of  Theorem \ref{Hig} are satisfied and $t_{0}\neq 0$. Then for all $f\in \mathbf{B}_{\pi}(D),\>\>g^{*}\in E^{*}$ and every $t\in \mathbb{R}$  the following formulas hold
\begin{equation}\label{s3}
\left<e^{tD}f,\>g^{*}\right>=
  \left<f,\>g^{*}\right>+t\sum_{n\in \mathbb{Z}}\frac{\left<e^{t_{n}D}f,\>g^{*}\right>-\left<f,\>g^{*}\right>}{t_{n}}\frac{G(t)}{G^{'}(t_{n})(t-t_{n})},
\end{equation}
and 
\begin{equation}\label{l3000}
\left<f,g^{*}\right>=\left<e^{t D}f,\>g^{*}\right>+t\sum_{n\in \mathbb{Z}}\frac{\left<e^{(t_{n}-t)D}f,\>g^{*}\right>-\left<e^{t D}f,\>g^{*}\right>}{t_{n}}\frac{G(-t)}{G^{'}(t_{n})(t+t_{n})}
\end{equation}
uniformly on all compact subsets of $\mathbb{R}$.
\end{theorem}

\begin{remark}

The last  formula represents a "measurement" $\left<f,\>g^{*}\right>$ through  "measurements" $\left<e^{(t_{n}-t)D}f,\>g^{*}\right>$ and $\left<e^{t D}f,\>g^{*}\right>$ which are different from $\left<f,\>g^{*}\right>$.

\end{remark}
       \begin{proof}
       
         If $f\in \mathbf{B}_{\sigma}(D)$ then for any $g^{*}\in E^{*}$ the function $F(t)=\left<e^{tD}f,\>g^{*}\right>$ belongs to $B_{\sigma}^{\infty}(\mathbb{R})$.
We consider $F_{1}\in B_{\sigma}^{2}( \mathbb{R}),$ which is defined as follows.
If $t\neq 0$ then 
$$
F_{1}(t)=\frac{F(t)-F(0)}{t}=\left<\frac{e^{tD}f-f}{t},\>g^{*}\right>,\>\>g^{*}\in E^{*}, 
$$
and if $t=0$ then
$$
F_{1}(t)=\frac{d}{dt}F(t)|_{t=0}=\left<Df,\>g^{*}\right>.
$$
We have
$$
F_{1}(t)=\sum_{n\in \mathbb{Z}}F_{1}(t_{n})\frac{G(t)}{G^{'}(t_{n})(t-t_{n})}
$$
or
$$
\left<\frac{e^{tD}f-f}{t},\>g^{*}\right>=\sum_{n\in \mathbb{Z}}\left<\frac{e^{t_{n}D}f-f}{t_{n}},\>g^{*}\right>\frac{G(t)}{G^{'}(t_{n})(t-t_{n})}
$$
uniformly on all compact subsets of $\mathbb{R}$. 

If  we pick a non-zero $\tau$ such that $\tau \neq t_{n}$ for all $t_{n}$ and set $f$ in (\ref{s3})  to $e^{\tau D}f$ then for $t=-\tau$ we will have the following formula which does not have vector $f$ on the right-hand side 
\begin{equation}\label{l3}
\left<f,g^{*}\right>=\left<e^{\tau D}f,\>g^{*}\right>+\tau\sum_{n\in \mathbb{Z}}\frac{\left<e^{(t_{n}-\tau)D}f,\>g^{*}\right>-\left<e^{\tau D}f,\>g^{*}\right>}{t_{n}}\frac{G(-\tau)}{G^{'}(t_{n})(\tau+t_{n})}.
\end{equation}
Theorem is proved.

\end{proof}
In \cite{S} the following  result can be found.

\begin{theorem}\label{S}
Let $\{t_{n}\}_{n\in \mathbb{Z}}$ be a sequence of real numbers such that 
$$
\sup_{n\in\mathbb{Z}}|t_{n}-n|<1/4.
$$
Define the entire function 
$$
G(z)=(z-t_{0})\prod_{n=1}^{\infty}\left(1-\frac{z}{t_{n}}\right)\left(1-\frac{z}{t_{-n}}\right).
$$
Let $\delta$ be any positive number such that $0<\delta<\pi$. Then for all $f\in \mathbf{B}_{\pi-\delta}(\mathbb{R})$ we have
$$
f(z)=\sum_{n\in \mathbb{Z}}f(t_{n})\frac{G(z)}{G^{'}(t_{n})(z-t_{n})}
$$
uniformly on all compact subsets of $\mathbb{C}$.
\end{theorem}

From here we obtain the next fact.

 \begin{theorem}\label{irregth}
 Suppose that  $D$ generates a one-parameter strongly continuous  group  $e^{tD}$ of isometries  in a Banach space $E$. Then in the same notations as in Theorem \ref{S} one has that for all $f\in \mathbf{B}_{\pi-\delta}(D)$ and all $g^{*}\in E^{*}$  the following formula holds 
\begin{equation}\label{s4}
 \left<e^{zD}f, g^{*}\right>=\sum_{n\in \mathbb{Z}}\left<e^{t_{n}D}f,\>g^{*}\right>\frac{G(z)}{G^{'}(t_{n})(z-t_{n})}
 \end{equation}
 uniformly on all compact subsets of $\mathbb{C}$.
\end{theorem}

\begin{proof}
Proof follows from Theorem \ref{S} \rm since for any $g^{*}\in E^{*}$ the function $ \left<e^{zD}f, g^{*}\right>$ belongs to $ \mathbf{B}_{\pi-\delta}(\mathbb{R})$.
\end{proof}
Note that if in the last formula we will set  $z$ to zero and assume that $t_{0}\neq 0$ we will have a representation  of $\left<f,\>g^{*}\right>$ in terms of samples $\left<e^{t_{n}D}f,\>g^{*}\right>\neq \left<f,\>g^{*}\right>$ i.e.
\begin{equation}\label{l2}
 \left<f, g^{*}\right>= -\sum_{n\in \mathbb{Z}}\left<e^{t_{n}D}f,\>g^{*}\right>\frac{G(0)}{G^{'}(t_{n})t_{n}}.
 \end{equation}

\section{An application to abstract Schr\"{o}dinger equation}\label{appl}

We now assume that $E$ is a Hilbert space and $D$ is a selfadjoint operator. Then $e^{itD}$ is one-parameter group of isometries of $E$.  
By  the spectral theory
\cite{BS}, there exist a direct integral of Hilbert spaces $A=\int
A(\lambda )dm (\lambda )$ and a unitary operator $\mathcal{F}_{D}$
from $E$ onto $A$, which transforms the domain $\mathcal{D}_{k}$
of the operator $ D^{k}$ onto $A_{k}=\{a\in A|\lambda^{k}a\in A
\}$ with norm

$$\|a(\lambda )\|_{A_{k}}= \left (\int^{\infty}_{-\infty} \lambda ^{2k}
\|a(\lambda )\|^{2}_{A(\lambda )} dm(\lambda ) \right )^{1/2} $$
and  $\mathcal{F}_{D}(Df)=\lambda (\mathcal{F}_{D}f), f\in
\mathcal{D}_{1}. $

In this situation one can prove \cite{Pes00} the following.

\begin{theorem}
A vector $f\in E$ belongs to $ \mathbf{B}_{\sigma}(D)$ if and only if support of $\mathcal{F}_{D}f$ is in $[-\sigma, \>\sigma]$.
\end{theorem}
We consider an  abstract Cauchy problem for a self-adjoint operator $D$ which consists of funding an abstract-valued function $u: \mathbf{R}\rightarrow E$ which satisfies Schr\"{o}dinger equation and has bandlimited initial condition
\begin{equation}\label{C1}
\frac{du(t)}{dt}=iDu(t),\>\>\>u(0)=f\in \mathbf{B}_{\sigma}(D)
\end{equation}
(see \cite{BB}, \cite{K} for more details). 

In this case formula (\ref{l2}) can be treated as a solution to inverse problem associated with (\ref{C1}).

\begin{theorem}
If conditions of Theorem \ref{irregth} are satisfied and  $t_{0}\neq 0$ then initial condition $f\in \mathbf{B}_{\pi-\delta}(D),\>\>0<\delta<\pi,$ in (\ref{C1}) can be reconstructed (in weak sense) from the values of the solution $u(t_{n})$ by using the formula 
\begin{equation}
 \left<f, g^{*}\right>=-\sum_{n\in \mathbb{Z}}\left<u(t_{n}), g^{*}\right>\frac{G(0)}{G^{'}(t_{n})t_{n}},\>\>\>g^{*}\in E^{*}.
 \end{equation}

\end{theorem}

Similar results can be formulated by using  formulas  (\ref{l1}) and (\ref{l3}).

\end{document}